\documentclass[12pt]{article}

\usepackage{amsthm,amsmath,amssymb,enumerate}
\usepackage{url}

\numberwithin{equation}{section}

\usepackage{tikz}
\usetikzlibrary{intersections, calc, arrows.meta}

\newtheorem{thm}{Theorem}[section]

\newtheorem{prop}[thm]{Proposition}

\newtheorem{lem}[thm]{Lemma}

\newtheorem{cor}[thm]{Corollary}

\theoremstyle{definition}
\newtheorem*{rem}{Remark}
\newtheorem*{remark}{Remark}
\newtheorem{ex}[thm]{Example}
\newtheorem{example}[thm]{Example}
\newtheorem*{conjecture}{Conjecture}

\newcommand{\C}{\mathbb{C}}

\newcommand{\Hom}{\operatorname{Hom}}

\newcommand{\Aut}{\operatorname{Aut}}
\newcommand{\bs}{\boldsymbol}
\newcommand{\pii}{\operatorname{pi}}
\newcommand{\pit}{\widetilde{\operatorname{pi}}}
\newcommand{\GF}{\mathrm{GF}}
\newcommand{\Paley}{\mathrm{Paley}}
\newcommand{\Terw}{\mathcal{T}}

\title{Terwilliger algebras and some related algebras defined by finite connected simple graphs}
 \author{Akihide Hanaki\thanks{Supported by JSPS KAKENHI Grant Number JP17K05165.}\\
 \small Faculty of Science, Shinshu University,\\[-0.8ex]
 \small Matsumoto, 390-8621, Japan\\[-0.8ex] 
 \small\tt hanaki@shinshu-u.ac.jp\\
 \and
 Masayoshi Yoshikawa \thanks{Supported by JSPS KAKENHI Grant Number JP20K03557.}\\
 \small Department of Mathematics,\\[-0.8ex]
 \small Hyogo University of Teacher Education,\\[-0.8ex]
 \small Shimokume, Kato, Hyogo, 673-1494, Japan\\[-0.8ex]
 \small\tt myoshi@hyogo-u.ac.jp
 }

\date{}

\begin{document}
\maketitle

\begin{abstract}
  For a finite connected simple graph, the Terwilliger algebra is a matrix algebra generated by
  the adjacency matrix and idempotents corresponding to
  the distance partition with respect to a fixed vertex.
  We will consider algebras defined by two other partitions and the centralizer algebra of the stabilizer of the fixed vertex in the automorphism group of the graph.
  We will give some methods to compute such algebras and examples for various graphs.

\vspace{5mm}

\end{abstract}

\section{Introduction}\label{sec:intro}

In \cite{MR1203683, MR1210403, MR1229433},
Paul Terwilliger defined a Terwilliger algebra, originally called a subconstituent algebra,
for an association scheme.
Especially, it was well studied for $P$- and $Q$-polynomial association schemes.
In \cite{Suzuki-NOTE}, Terwilliger algebras for an arbitrary finite connected simple graph were also studied.
Recently, Shuang-Dong Li, Yi-Zheng Fan, Tatsuro Ito, Masoud Karimi, and Jing Xu \cite{MR4147627} dealt with Terwilliger algebras of trees motivated by the following conjecture.

\begin{conjecture}[J. Koolen]
  For almost all finite connected simple graphs, the Terwilliger algebras coincide with the full matrix algebras.
\end{conjecture}

After that, Jing Xu, Tatsuro Ito, and  Shuang-Dong Li \cite{Xu-Ito-Li} considered Terwilliger algebras and centralizer algebras of trees.
Then, they investigated when the two algebras coincide.

For a graph $\Gamma = (X,E)$ with the adjacency matrix $A$ and a fixed vertex $x_0 \in X$,
we will consider the following subalgebras $\Terw_\ell(\Gamma, x_0)$ ($\ell=0,1,2,3,4$) of $M_X(\C)$,
the full matrix algebra over the complex number field $\C$, rows and columns of whose matrices are indexed by the set $X$.
We denote by $E_x$ for $x\in X$ the matrix in $M_X(\C)$
whose $(x,x)$-entry is $1$ and the other entries are $0$,
and set $E_Y=\sum_{y\in Y}E_y$ for $Y\subset X$.
Let $G$ be the automorphism group of the graph $\Gamma$ and
$G_{x_0}$ the stabilizer of $x_0\in X$ in $G$.
Naturally, $G$ acts on $M_X(\C)$ by permuting rows and columns.

\begin{itemize}
  \item (the adjacency algebra)
  Set $\Terw_0(\Gamma, x_0) = \C \langle A \rangle$, the unital $\C$-subalgebra of $M_X(\C)$ generated by $A$.
  \item Set $\Terw_1(\Gamma, x_0) = \C \langle A, E_{x_0} \rangle$.
  \item (the Terwilliger algebra)
  Consider the distance partition of $X$ with respect to the vertex $x_0$:
  $X = X_0 \cup \dots \cup X_D$, where $X_k = \{ x\in X: \partial(x_0,x) = k\}$ and $D$ is the diameter of $\Gamma$ with respect to $x_0$, the maximal distance from $x_0$.
  Set $\Terw_2(\Gamma, x_0) = \C \langle A, E_{X_0}, E_{X_1}, \dots, E_{X_D} \rangle$.
  \item Let $Y_1, \dots, Y_r$ be the $G_{x_0}$-orbits on $X$.
  Set $\Terw_3(\Gamma, x_0) = \C \langle A, E_{Y_1}, \dots, E_{Y_r} \rangle$.
  \item (the centralizer algebra)
  Set $\Terw_4(\Gamma, x_0) = \{M\in M_X(\C): \text{$M^\sigma=M$ for any $\sigma \in G_{x_0}$}\}$.
\end{itemize}

The vertex $x_0$ will be called the \emph{base vertex}.
When there is not fear of the confusion, we omit $\Terw_\ell(\Gamma, x_0)$ with  $\Terw_\ell$.
Since the distances are preserved by automorphisms, we can see that 
\[
 \Terw_0 \subset \Terw_1 \subset \Terw_2 \subset \Terw_3 \subset \Terw_4 \subset M_X(\C).
\]
The aim of this paper is computing examples of these algebras for various simple graphs
and finding examples such that these algebras are different.
We remark that all $\Terw_\ell$ are semisimple, because they are closed by transpose and complex conjugate. 
Since we are considering semisimple algebras over an algebraically closed field, they are isomorphic to direct sums of full matrix algebras.
We also remark that the adjacency algebra $\Terw_0$ does not depend on $x_0$ and is commutative.

Since the algebras $\Terw_\ell$ ($\ell=0,1,2,3,4$) are defined as subalgebras of $M_X(\C)$, they act on $\C X$, the vector space with a formal basis $X$, by right multiplication.
The vector space $\C X$ will be called the \emph{standard} $\Terw_\ell$-module.

The conjecture by Koolen and the paper \cite{MR4147627} are about the relationship between $\Terw_2$ and $M_X(\C)$.
And the paper \cite{Xu-Ito-Li} gives a necessary and sufficient condition for $\Terw_2 = \Terw_4$ for trees.

This paper is organized as follows.
Section \ref{sec:prelim} describes notations and terminology on finite connected simple graphs and idempotents of semisimple algebras.
In Section \ref{sec:basic}, we list basic facts for $\Terw_\ell$ ($\ell = 0,1,2,3,4$).
In Section \ref{sec:T1}, we deal with the structure of $\Terw_1$.
Especially, we determine the structure of $\Terw_1$ for distance-regular graphs.
In Section \ref{sec:T2SRG}, we consider $\Terw_2$ for strongly regular graphs.
In Section \ref{sec:val1}, we consider the structure of $\Terw_2$ and $\Terw_3$ with respect to the base vertex of valency $1$.
Proposition \ref{prop:valencyone} given in this section is used for a construction of an infinite family of examples for $\Terw_2 \neq \Terw_3$ (Section \ref{sec:Gamma}).
In Section \ref{sec:someGRPH}, we show the structures of $\Terw_\ell$ for path graphs, star graphs and cycle graphs.
Finally, we give infinite families of examples for
$\Terw_3 \neq \Terw_4$ in Section \ref{sec:Paley} (Paley graphs) and for
$\Terw_2 \neq \Terw_3$ in Section \ref{sec:Gamma}.

\section{Preliminaries}\label{sec:prelim}
Let $\Gamma=(X,E)$ be a finite connected simple graph.
Let $M_X(\C)$ be the full matrix algebra over the complex number field $\C$,
rows and columns of whose matrices are indexed by the set $X$.
For $M\in M_X(\C)$, $M_{x,y}$ will be the $(x,y)$-entry of $M$.
For $x,y\in X$, we write $E_{x,y}$ for the matrix unit.
The identity matrix in $M_X(\C)$ will be denoted by $I$ or $I_X$.
The all-one matrix in $M_X(\C)$ will be denoted by $J$ or $J_X$.
The zero matrix in $M_X(\C)$ will be denoted by $O$ or $O_X$.
We write by $A = A(\Gamma) \in M_X(\C)$ the adjacency matrix of $\Gamma$,
namely, $A=\sum_{\{x,y\} \in E} \left( E_{x,y} + E_{y,x} \right)$.
For $Y \subset X$, $E_Y \in M_X(\C)$ will be the diagonal matrix whose diagonal entry $(E_Y)_{x,x}$ is $1$ if $x \in Y$ and $0$ otherwise.
Namely, $E_Y = \sum_{y \in Y} E_{y,y}$.
When $Y = \{y\}$, we write $E_y$ instead of $E_{\{y\}}$.
Let $G$ be the automorphism group of the graph $\Gamma$.
For $x \in X$, $G_x$ will be the stabilizer of $x$ in $G$.
Naturally, $G$ acts on $M_X(\C)$ by permuting rows and columns.

\subsection{Finite connected simple graphs}
Let $X$ be a finite set and $E$ be a subset of $2$-subsets of $X$.
In this case, we say that $\Gamma=(X,E)$ a \emph{finite simple graph}.
An element of $X$ is called a \emph{vertex} and an element of $E$ is called an \emph{edge}.
When $\{x,y\}\in E$, we say that $x$ and $y$ are \emph{adjacent} in $\Gamma$.
A \emph{walk} of \emph{length} $\ell$ in $\Gamma$ is a finite sequence $x_0,x_1,\dots,x_\ell$ in $X$ such that
$\{x_{i-1},x_i\}\in E$ for all $i=1,2,\dots,\ell$.
In this case, the walk is called a walk from $x_0$ to $x_\ell$.
Remark that a walk can contain the same vertices.
For $x,y \in X$, the \emph{distance} of $x$ and $y$ is
the minimal length of a walk from $x$ to $y$,
and denoted by $\partial(x,y)$.
If there is no walk from $x$ to $y$, we set $\partial(x,y) = \infty$.
A graph is said to be \emph{connected} if $\partial(x,y) < \infty$ for all $x,y\in X$.
The maximal distance will be called the \emph{diameter} of $\Gamma$.
For $x_0 \in X$, the \emph{diameter} of $\Gamma$ with respect to $x_0$ is the
maximal distance from $x_0$ : $\max\{\partial(x_0,y): y\in X\}$.

Let $\Gamma=(X,E)$ be a finite simple connected graph.
An \emph{automorphism} $\sigma$ of $\Gamma$ is a permutation on $X$ such that $\{x, y\} \in E$
if and only if $\{\sigma(x), \sigma(y)\} \in E$.
The \emph{automorphism group} $\Aut(\Gamma)$ is the group
consisting of all automorphisms of $\Gamma$.
For $x\in X$, the number of neighbors of $x$ is called the \emph{valency} of $x$.
A graph $\Gamma$ is said to be ($k$-)\emph{regular}
if all vertices have the same valency $k$. 
A graph $\Gamma$ is said to be \emph{vertex-transitive} 
if $\Aut(\Gamma)$ is transitive on $X$.
A graph $\Gamma$ is said to be \emph{distance-transitive} 
if, for any $x_1, x_2, y_1, y_2 \in X$ such that $\partial(x_1,x_2) = \partial(y_1,y_2)$,
there exists $\sigma \in \Aut(\Gamma)$ such that $\sigma(x_1) = y_1$ and $\sigma(x_2) = y_2$.
A graph $\Gamma$ is said to be \emph{distance-regular} 
if, for $i,j,k = 0, 1, \dots, D$, where $D$ is the diameter of $\Gamma$, there exists a
non-negative integer $p_{ij}^k$ such that 
$\sharp\{z \in X: \partial(x,z) = i,\ \partial(z,y)=j\}=p_{ij}^k$
for any $x,y \in X$ with $\partial(x,y)=k$.
Clearly, a distance-transitive graph is vertex-transitive and distance-regular.

A $k$-regular graph with $n$ vertices is said to be \emph{strongly regular}, with parameters $(n, k, \lambda, \mu)$ if every adjacent vertices have $\lambda$ common neighbors, and every non-adjacent vertices have $\mu$ common neighbors.
A strongly regular graph is called \emph{primitive} if both of the graph and its complement are connected.
In this article, when we say a strongly regular graph, we mean a primitive strongly regular graph.

The \emph{adjacency matrix} $A$ of a finite simple connected graph $\Gamma$ is a matrix in $M_X(\C)$ such that
$A_{x,y}=1$ if $\{x,y\}\in E$ and $0$ otherwise.
For the adjacency matrix $A$ and a non-negative integer $k$, the $(x,y)$-entry of $A^k$
is the number of walks of length $k$ from $x$ to $y$ \cite[Lemma 2.5]{MR1271140}.
Let $D$ be the diameter of $\Gamma$.
Then $A$ has at least $D+1$ distinct eigenvalues \cite[Corollary 2.7]{MR1271140}.
Remark that all eigenvalues of $A$ are real numbers since $A$ is a real symmetric matrix.
An eigenvalue of $A$ is also called an \emph{eigenvalue} of the graph $\Gamma$.

\subsection{Idempotents of semisimple algebras}

In this subsection, we will summarize basic facts on idempotents of algebras.
For details, see \cite{MR998775}, for example.

Let $K$ be a field and $\mathcal{A}$ a finite dimensional $K$-algebra.
In this article an $\mathcal{A}$-module means a finite dimensional right $\mathcal{A}$-module.
A nonzero element $e$ of $\mathcal{A}$ is called an \emph{idempotent} if $e^2=e$.
Two idempotents $e$ and $f$ are said to be \emph{orthogonal} if $ef=fe=0$.
An idempotent $e$ is said to be \emph{primitive} if $e$ is not expressed as a sum of two orthogonal idempotents. 
We say that $e = e_1 + \dots + e_r$ is an \emph{idempotent decomposition} of $e$ if all $e_i$ are idempotents and they are orthogonal to each other.
An idempotent decomposition $e = e_1 + \dots + e_r$ is said to be a \emph{primitive idempotent decomposition} if all $e_i$ are primitive.
Let $e = e_1 + \dots + e_r$ be an idempotent decomposition of $e$.
Then we have a direct sum decomposition $e \mathcal{A} = e_1 \mathcal{A} \oplus \dots \oplus e_r \mathcal{A}$ as an $\mathcal{A}$-module.
Conversely, a direct sum decomposition of $e \mathcal{A}$ induces an idempotent decomposition.
Thus an idempotent $e$ is primitive if and only if $e \mathcal{A}$ is an indecomposable $\mathcal{A}$-module.

By $\pii(\mathcal{A})$, we denote the set of all primitive idempotents of $\mathcal{A}$.
For $e, f \in \pii(\mathcal{A})$, we define $e \sim f$ if $e \mathcal{A} \cong f \mathcal{A}$ as $\mathcal{A}$-modules.
Clearly, this is an equivalence relation on $\pii(\mathcal{A})$.
By $\pit(\mathcal{A})$, we denote the set of all equivalence classes of $\pii(\mathcal{A})$, and by $[e]$ the equivalence class containing $e\in\pii(\mathcal{A})$.
For an $\mathcal{A}$-module $V$ and an idempotent $e$ of $\mathcal{A}$, we have $\Hom_\mathcal{A}(V, e\mathcal{A})\cong Ve$ as $K$-spaces.
Thus, for $e, f\in \pii(\mathcal{A})$,  $e\sim f$ if and only if $e\mathcal{A} f\ne 0$.

\begin{lem}\label{lem:equividemp}
  Let $\mathcal{A}$ be a finite dimensional algebra over a field $K$.
  For an idempotent $e$ of $\mathcal{A}$, $e \mathcal{A} e$ is a $K$-algebra with the identity element $e$.
  For an idempotent $f$ of $e \mathcal{A} e$, $f$ is primitive in $e \mathcal{A} e$ if and only if $f$ is primitive in $\mathcal{A}$.
  For $f,f'\in\pii(e\mathcal{A} e)$, $f$ and $f'$ are equivalent in $e \mathcal{A} e$ if and only if they are equivalent in $\mathcal{A}$.
\end{lem}

\begin{proof}
  If $f$ is primitive in $\mathcal{A}$, then clearly it is primitive in $e\mathcal{A} e$.
  Suppose that $f$ is not primitive in $\mathcal{A}$,
  $f = g + g'$ is an idempotent decomposition of $f \in e \mathcal{A} e$ in $\mathcal{A}$.
  Since $f \in e \mathcal{A}e$, it follows that $f = fe = ef$.
  Then $gf = g(g+g') = g^2 = g$, and similarly $fg = g$ and $g'f = fg' = g'$.
  This means that $g = fgf = ege \in e\mathcal{A} e$ and $g' \in e \mathcal{A} e$.
  Thus $f$ is not primitive in $e \mathcal{A} e$.

  Suppose $f,f'\in \pii(e\mathcal{A} e)$.
  If $f$ and $f'$ are equivalent in $e \mathcal{A} e$, then clearly they are equivalent in $\mathcal{A}$.
  Suppose that $f$ and $f'$ are equivalent in $\mathcal{A}$.
  Since $f,f'\in e\mathcal{A} e$, $f=fe$ and $f'=ef'$, we have
  $0\ne f\mathcal{A} f'=f(e\mathcal{A} e)f'$ and so they are equivalent in $e\mathcal{A} e$.
\end{proof}

Every idempotent of a finite dimensional algebra has a primitive idempotent decomposition.

\begin{lem}
Let $e=e_1+\dots+ e_r$ be an idempotent decomposition of an idempotent $e$,
and $e_i=f_1^{(i)}+\dots+f_{r_i}^{(i)}$ primitive idempotent decompositions of $e_i$, $i=1,\dots,r$.
Then $e=\sum_{i=1}^r\sum_{j=1}^{r_i}f_{j}^{(i)}$ is a primitive idempotent decomposition of $e$.
\end{lem}

\begin{proof}
  All $f_j^{(i)}$ are primitive idempotents.
  It is enough to show that the orthogonality of them.
  Suppose $(i,j)\ne (i'j')$.
  If $i=i'$, then $f_j^{(i)}f_{j'}^{(i)}=f_{j'}^{(i)}f_{j}^{(i)}=0$
  since $e_i=f_1^{(i)}+\dots+f_{r_i}^{(i)}$ is a  primitive idempotent decomposition.
  Suppose that $i\ne i'$.
  We know $f_{j}^{(i)}=e_if_{j}^{(i)}e_i$ and $f_{j'}^{(i')}=e_{i'}f_{j'}^{(i')}e_{i'}$.
  Thus we have $f_j^{(i)}f_{j'}^{(i')}=e_if_{j}^{(i)}e_ie_{i'}f_{j'}^{(i')}e_{i'}=0$ and 
  $f_{j'}^{(i')}f_{j}^{(i)}=e_{i'}f_{j'}^{(i')}e_{i'}e_{i}f_{j}^{(i)}e_{i}=0$.
\end{proof}

Now we suppose that $K=\C$ the complex number field and $\mathcal{A}$ is semisimple.
In this case, $\mathcal{A}\cong \bigoplus_{i=1}^s M_{n_i}(\C)$ for some positive integers $n_i$ by Wedderburn's Theorem \cite[Theorem 1.8.5]{MR998775}.
The projections $\pi_i:\mathcal{A}\to M_{n_i}(C)$, $i=1.\dots,s$, are representatives of equivalent classes of irreducible representations of $\mathcal{A}$.
Also there is a bijection between $\pit(\mathcal{A})$ and the set of isomorphism classes of simple $\mathcal{A}$-modules by $[e]\mapsto e\mathcal{A}$.
Thus every primitive idempotent belongs to exactly one competent $M_{n_i}(\C)$.
A primitive idempotent of $\mathcal{A}$ is similar to a diagonal matrix unit of some $M_{n_i}(\C)$.
Let $1=e_1+\dots+e_n$ be a primitive idempotent decomposition of $1$ in $\mathcal{A}$.
Then $n_i$ is the number of $e_j$'s belonging to $M_{n_i}(\C)$ and
is equal to the dimension of the corresponding simple module.

\section{Basic facts}\label{sec:basic}

In this section, we will state some basic facts.
We keep the notations in Introduction.
We suppose that $\Gamma=(X,E)$ is a finite simple connected graph and fix $x_0\in X$.
The automorphism group $G=\Aut(\Gamma)$ also acts on $X\times X$ by $\sigma(x,y)=(\sigma(x),\sigma(y))$.

\begin{thm} \cite[p.10]{MR1271140} \label{basic0}
  Let $\Gamma$ be a finite simple connected graph whose adjacency matrix has $t$ distinct eigenvalues.
  Then it follows that $\dim \Terw_0(\Gamma, x_0) = t$ and $\Terw_0(\Gamma, x_0) \cong \bigoplus_{i=1}^{t} \C$.
\end{thm}

\begin{lem}\label{basic1}
  If there is $\sigma\in G$ such that $\sigma(x_0)=y_0$,
  then $\Terw_\ell(\Gamma,x_0)\cong \Terw_\ell(\Gamma,y_0)$.
  If $\Gamma$ is vertex-transitive, then the structure of $\Terw_\ell(\Gamma, x_0)$ ($\ell=0,1,2,3,4$)
  does not depend on the base vertex $x_0$.
\end{lem}

\begin{lem}\label{basic2}
  If $\Gamma$ is distance-transitive, then $\Terw_2(\Gamma,x_0) = \Terw_3(\Gamma,x_0)$.
\end{lem}

\begin{proof}
  In this case, the distance partition with respect to the base vertex $x_0$ is just the set of $G_{x_0}$-orbits.
\end{proof}

\begin{lem}\label{basic4}
  Let $Y_1,\dots,Y_r$ be $G_{x_0}$-orbits on $X\times X$.
  Set $A_{Y_i} = \sum_{(x,y) \in Y_i} E_{x,y}$.
  Then $\{A_{Y_1}, \dots, A_{Y_r}\}$ spans $\Terw_4(\Gamma,x_0)$ and $\dim \Terw_4(\Gamma,x_0)=r$. 
\end{lem}

\begin{proof}
  This is shown by direct calculations.
\end{proof}

\begin{lem}\label{basic3}
  The following statements are equivalent.
  \begin{enumerate}[(1)]
    \item $\Terw_3(\Gamma,x_0)=M_X(\C)$.
    \item $\Terw_4(\Gamma,x_0)=M_X(\C)$.
    \item $G_{x_0}=1$.
  \end{enumerate}
\end{lem}

\begin{proof}
  It is clear that (1) implies (2) and (3) implies (1).
  Also (2) implies (3), by Lemma \ref{basic4}, for example.
\end{proof}

In \cite[Theorem 4]{MR4147627}, it is shown that $G_{x_0}=1$ if and only if $\Terw_2(\Gamma,x_0)=M_X(\C)$ for trees.
This is not true, in general. For example, see Example \ref{ex4.3}.

\begin{lem}\label{lem:principal}
  Let $X=Y_1\cup\dots \cup Y_r$ be an arbitrary partition of the vertex set $X$ of a finite simple connected graph $\Gamma=(X,E)$.
  Suppose that $|Y_i|=1$.
  We set $\Terw=\C\langle A, E_{Y_1}, \dots, E_{Y_r} \rangle$.
  Then the dimension of the simple $\Terw$-module $E_{Y_i}\Terw$ is at least $r$.
  Moreover, if $|Y_i|=|Y_j|=1$, then $E_{Y_i}\Terw\cong E_{Y_j}\Terw$.
\end{lem}

\begin{proof}
  By the connectivity of $\Gamma$, $E_{Y_i}\Terw E_{Y_k}\ne 0$ for all $1\leq k\leq r$.
  Since $|Y_i|=1$, $E_{Y_i}$ is a primitive idempotent and every $E_{Y_k}$ contains an idempotent
  equivalent to $E_{Y_i}$.
  Thus $E_{Y_i}$ has at least $r$ equivalent idempotents in the primitive idempotent decomposition of $1$.
  The last statement is clear.
\end{proof}

For the distance partition $X=X_0\cup\dots\cup X_D$ with respect to the vertex $x_0$, $X_0=\{x_0\}$ and $|X_0|=1$. 
We can apply Lemma \ref{lem:principal} and the simple module $E_{X_0}\Terw_2$ is called the \emph{principal module}.

\section{The structure of $\Terw_1$}\label{sec:T1}

In this section, we consider the structure of $\Terw_1 =\C \langle A, E_{x_0} \rangle$.

The next lemma is clearly holds.

\begin{lem}
  For $\bs{v} = (v_x)_{x\in X}$, $\bs{v} E_{x_0}= \bs{0}$ if and only if $v_{x_0} = 0$.
\end{lem}

\begin{prop}\label{one1}
  Let $\Gamma = (X, E)$ be a simple connected graph with $|X| = n$.
  Fix a base vertex $x_0 \in X$.
  Suppose that $\lambda$ is an eigenvalue of the adjacency matrix $A$ of $\Gamma$, and
  $\bs{v}$ is a corresponding eigenvector.
  If the $x_0$-th entry of $\bs{v}$ is zero, then $\C\bs{v}$ is a simple $\Terw_1(\Gamma, x_0)$-module
  not isomorphic to $E_{x_0} \Terw_1(\Gamma, x_0)$ as a $\Terw_1(\Gamma, x_0)$-module.
  If $\bs{w}$ is an eigenvector corresponding to $\mu$ different from $\lambda$ with the $x_0$-th entry zero,
  then $\C \bs{w}$ is non-isomorphic to $\C \bs{v}$ as a $\Terw_1(\Gamma, x_0)$-module.
\end{prop}

\begin{proof}
  Since $\bs{v}$ is an eigenvector of $A$, $\bs{v} A = \lambda \bs{v} \in \C \bs{v}$.
  By assumption, $\bs{v} E_{x_0} = \bs{0}$.
  By $\Terw_1 = \C \langle A, E_{x_0} \rangle$,  $\C \bs{v}$ is a $\Terw_1$-module.
  The idempotent $E_{x_0}$ acts on $E_{x_0} \Terw_1$ as nonzero and on $\C \bs{v}$ as $0$.
  Thus $\C \bs{v}$ is not isomorphic to $E_{x_0} \Terw_1$.
  The modules  $\C \bs{v}$ and $\C \bs{w}$ are non-isomorphic since the actions of $A$ are different.
\end{proof}

\begin{cor}\label{one2}
  Suppose that an eigenvalue $\lambda$ of $A$ has the eigenspace $V_\lambda$ of dimension $d_{\lambda}$.
  Set $W_\lambda := \{\bs{v} \in V_\lambda: \bs{v} E_{x_0} = \bs{0}\}$.
  Then $\dim W_\lambda = d_{\lambda} -1$ or $d_{\lambda}$ and $W_\lambda$
  is a direct sum of isomorphic $1$-dimensional simple $\Terw_1(\Gamma, x_0)$-modules.
\end{cor}

\begin{proof}
  Since $\{\bs{v} \in \C X: \bs{v} E_{x_0} = \bs{0}\}$ has dimension $n-1$, $\dim W_\lambda = d_{\lambda} -1$ or $d_{\lambda}$.
  For an arbitrary basis $\{ \bs{v}_i \}$ of $W_\lambda$,
  $\C\bs{v}_i$'s are isomorphic simple $\Terw_1$-modules
  by Proposition \ref{one1} and $W_\lambda=\bigoplus_i \C\bs{v}_i$.
\end{proof}

\begin{prop}\label{one3}
  Set $D = \max\{\partial(x_0,y): y \in X\}$, the diameter with respect to $x_0$.
  Then $\dim E_{x_0} \Terw_1(\Gamma, x_0)\geq D+1$.
\end{prop}

\begin{proof}
  For $i=0,\dots, D$, we consider $E_{x_0}A^i$.
  The matrix $E_{x_0}A^i$ contains $E_{x_0,y}$ with $\partial(x_0,y)=i$ and no $E_{x_0,y}$ with $\partial(x_0,y)>i$. 
  Thus $\{E_{x_0}A^i: i=0,1,\dots,D\}$ is a linearly independent set, and so $\dim E_{x_0} \Terw_1 \geq D+1$.
\end{proof}

The next result can be applied to $\Terw_1$ for distance-regular graphs.

\begin{thm}\label{one4}
  Let $\Gamma = (X,E)$ be a simple connected graph with diameter $D$ with respect to the base vertex $x_0$.
  Suppose that the graph has exactly $D+1$ eigenvalues $\lambda_0, \dots, \lambda_D$
  with multiplicities $m_0, \dots, m_D$, respectively.
  Set $t := \sharp\{i: m_i>1\}$.
  Then $\Terw_1(\Gamma, x_0) \cong M_{D+1}(\C) \oplus \bigoplus_{i=1}^t \C$.
  Especially, $\dim \Terw_1(\Gamma, x_0) = (D+1)^2 + t$.
\end{thm}

\begin{proof}
  By Proposition \ref{one1} and Corollary \ref{one2},
  there are $t$ non-isomorphic simple $\Terw_1$-modules $S_i$ for $m_i>1$.
  Set $m_i'$ the multiplicity of $S_i$ in $\C X$.
  Note that $m_i'=m_i$ or $m_i-1$.

  We consider $E_{x_0} \Terw_1$.
  By Proposition \ref{one3}, $\dim E_{x_0} \Terw_1 \geq D+1$.
  Since the rank of $E_{x_0}$ is one, the multiplicity of $E_{x_0} \Terw_1$ in $\C X$ is one.
  Now, considering that $\C X$ contains a submodule isomorphic to
  $E_{x_0} \Terw_1 \oplus \bigoplus_{m_i>1} m_i' S_i$, we have
\[
 n = \sum_{i=0}^D m_i = (D+1) + \sum_{i=0}^D (m_i-1) \leq \dim E_{x_0} \Terw_1 + \sum_{m_i>1} m_i' \leq n.
\]
  This shows that $\dim E_{x_0} \Terw_1 = D+1$ and
  $\{E_{x_0} \Terw_1\} \cup \{S_i: m_i>1\}$ is the set of all representatives of simple $\Terw_1$-modules.
\end{proof}

\section{Terwilliger algebras $\Terw_2$ of strongly regular graphs}\label{sec:T2SRG}

The next proposition is essentially proved in \cite[Theorem 5.1]{MR523457}.
See also \cite{MR1294534}.

\begin{prop}\label{SRG}
  Let $\Gamma=(X,E)$ be an $(n,k,\lambda,\mu)$-strongly regular graph.
  Fix a vertex $x_0\in X$.
  Let $X=X_0\cup X_1\cup X_2$ be the distance partition of $X$ with respect to $x_0$.
  Then $\C \langle E_{X_1} A E_{X_1}, E_{X_1} J_X E_{X_1}\rangle = E_{X_1} \Terw_2(\Gamma, x_0) E_{X_1}$.
  Especially,  $E_{X_1} \Terw_2(\Gamma, x_0) E_{X_1}$ is a commutative algebra.
  (The same things hold also for $X_2$.)
\end{prop}

\begin{proof}
  It is clear that
  $\C \langle E_{X_1} A E_{X_1}, E_{X_1} J_X E_{X_1} \rangle \subset E_{X_1}\Terw_2(\Gamma, x_0)E_{X_1}$.

  For convenience, we write $A_{(i,j)}$ for $E_{X_i}AE_{X_j}$, $J_{(i,j)}$ for $E_{X_i}J_XE_{X_j}$,
  and $I_{(i,j)}$ for $E_{X_i}I_XE_{X_j}$.
  Remark that $A_{(1,1)}J_{(1,1)}=J_{(1,1)}A_{(1,1)}=\lambda J_{(1,1)}$ and
  $A_{(2,2)}J_{(2,2)}=J_{(2,2)}A_{(2,2)}=(k-\mu) J_{(2,2)}$, and so on.

  Since $E_{X_1} \Terw_2 E_{X_1}$ is generated by $E_{X_1} A E_{X_{s_1}} A E_{X_{s_2}} \cdots E_{X_{s_t}} A E_{X_1}$ for $0 \leq s_u \leq 2$, it is enough to show that such terms are in $\C \langle A_{(1,1)}, J_{(1,1)} \rangle$.
  We may assume that $s_u \neq 1$ for all $u$
  because we can divide the term into two parts if $s_u = 1$ for some $u$.
  Since $A_{(0,0)} = A_{(0,2)} = A_{(2,0)} = O$, this term is
  $E_{X_1} A E_{X_0} A E_{X_1} = J_{(1,1)}$ or $O$ if $s_u = 0$ for some $u$.
  Moreover, if $s_u = 2$ for some $u$, then this term is
  $A_{(1,2)} {A_{(2,2)}}^m A_{(2,1)}$ ($m = 0, 1, \dots$) or $O$.
Thus, we can see that $E_{X_1}\Terw_2(\Gamma, x_0)E_{X_1}$ is generated by
$A_{(1,1)}, A_{(1,2)}{A_{(2,2)}}^m A_{(2,1)}$ ($m=0,1,\dots$) and $J_{(1,1)}$.
  It is clear that $J_{(1,1)}, {A_{(1,1)}} \in \C \langle A_{(1,1)}, J_{(1,1)}\rangle$.
  We will show that $A_{(1,2)}{A_{(2,2)}}^m A_{(2,1)} \in \C \langle A_{(1,1)}, J_{(1,1)}\rangle$.

  By a known equation $A^2=kI+\lambda A+\mu (J-A-I)$ for an $(n,k,\lambda,\mu)$-strongly regular graph, we have
  \begin{eqnarray}
    A_{(1,1)}A_{(1,2)}+A_{(1,2)}A_{(2,2)} &=& (\lambda-\mu) A_{(1,2)}+\mu J_{(1,2)} \label{eq:srg1}\\
    J_{(1,1)}+A_{(1,1)}A_{(1,1)}+A_{(1,2)}A_{(2,1)} &=& (k-\mu) I_{(1,1)}+(\lambda-\mu) A_{(1,1)}+\mu J_{(1,1)}.\label{eq:srg2}
  \end{eqnarray}
  Since $\mu J_{(1,2)}= J_{(1,1)}A_{(1,2)}$, the equation (\ref{eq:srg1}) shows that
  $A_{(1,2)}A_{(2,2)}=g(A_{(1,1)},J_{(1,1)})A_{(1,2)}$ for some polynomial $g(x, y)$.
  Thus ${A_{(1,2)}A_{(2,2)}}^m = g(A_{(1,1)},J_{(1,1)})^m A_{(1,2)}$.
  The equation (\ref{eq:srg2}) shows that $A_{(1,2)}A_{(2,1)}=h(A_{(1,1)}, J_{(1,1)})$  for some polynomial $h(x,y)$.
  Now $A_{(1,2)}{A_{(2,2)}}^m A_{(2,1)}=g(A_{(1,1)},J_{(1,1)})^m h(A_{(1,1)}, J_{(1,1)})\in \C\langle A_{(1,1)}, J_{(1,1)}\rangle$.
  The proof is completed.  
\end{proof}

\begin{prop}\label{dimA2}
  Let $\Gamma=(X,E)$ be an $(n,k,\lambda,\mu)$-strongly regular graph.
  Fix a vertex $x_0\in X$.
  Then $\dim \Terw_2(\Gamma, x_0)\leq 2n+3$.
\end{prop}

\begin{proof}
  Set $d_i=\dim E_{X_i}\Terw_2 E_{X_i}$ for $i=1,2$.
  Since $E_{X_i}\Terw_2 E_{X_i}$ is a commutative semisimple algebra, $d_1\leq k$, $d_2\leq n-k-1$, and $E_{X_i}$ is a sum of $d_i$ non-equivalent primitive idempotents.
  Thus $E_{X_i}\Terw_2$ is a direct sum of $d_i$ non-isomorphic simple $\Terw_2$-modules.
  Now $E_{X_1}\Terw_2E_{X_2}\cong \Hom_{\Terw_2}(E_{X_1}\Terw_2, E_{X_2}\Terw_2)$ and thus $\dim E_{X_1}\Terw_2E_{X_2}\leq \min \{d_1,d_2\}$.
  Similarly we have  $\dim E_{X_2}\Terw_2E_{X_1}\leq \min\{d_1,d_2\}$.
  We also remark that $\dim E_{X_0}\Terw_2E_{X_i}=\dim E_{X_i}\Terw_2E_{X_0}=1$ for $i=0,1,2$.
  By $\dim \Terw_2=\sum_{0\leq i, j\leq 2}\dim E_{X_1}\Terw_2E_{X_2}$, we have
  $$\dim\Terw_2\leq5+d_1+d_2 + 2 \min\{d_1,d_2\} \leq 5+(n-1)+(n-1)=2n+3$$
  and the assertion holds.
\end{proof}

Our algebra $\Terw_2(\Gamma, x_0)$ is just a Terwilliger algebra.
A simple $\Terw_2(\Gamma, x_0)$-module $W$ is said to be \emph{thin} if $\dim WE_{X_i}\leq 1$ for all $i=0,1,\dots,D$
\cite[Section 3]{MR1203683}.
A (distance-regular) graph is said to be \emph{thin} with respect to $x_0$
if every irreducible $\Terw_2(\Gamma, x_0)$-module is thin.
Now $\dim WE_{X_i}=\dim \Hom(E_{X_i}\Terw_2(\Gamma, x_0), W)$ and the condition $\dim WE_{X_i}\leq 1$
means that the modules $E_{X_i}\Terw_2(\Gamma, x_0)$ contain at most one simple submodule isomorphic to $W$.
This condition is satisfied for all simple modules $W$ if $E_{X_i}\Terw_2(\Gamma, x_0)E_{X_i}$ ($i=1,2$) are commutative.
As a consequent of Proposition \ref{SRG}, we can say that every strongly regular graph is thin with respect to any vertex.
This fact was proved in  \cite[Lemma 3.3]{MR1294534}.

\begin{ex}\label{ex4.3}
  It is known that there are many strongly regular graphs with the trivial automorphism groups.
  By Lemma \ref{basic3}, $\dim \Terw_3=\dim \Terw_4=n^2$ for them.
  By Proposition \ref{dimA2}, $\dim \Terw_2\leq 2n+3$.
  Thus $\Terw_2\subsetneq \Terw_3$ holds for them if $n\geq 4$.
\end{ex}

\section{A base vertex of valency one}\label{sec:val1}

In this section, we consider the structures of $\Terw_2$ and $\Terw_3$ in the case that the valency of the base vertex $x_0$ is one.
There exists the unique neighbor $x_1$ of $x_0$.
We set $X' = X \setminus \{x_0\}$, $E' = E \setminus \{\{x_0,x_1\}\}$, and consider the graph $\Gamma' = (X', E')$ with the base vertex $x_1$.
Naturally, we can regard $M_{X'}(\C)$ as a subset of $M_X(\C)$.

\begin{lem}\label{lem:valencyone}
  $E_{X'} \Terw_\ell(\Gamma, x_0) E_{X'} = \Terw_\ell(\Gamma', x_1)$ for $\ell = 2, 3$. 
\end{lem}
\begin{proof}
  We consider the case $\ell=2$.
  Let $X = X_0 \cup X_1 \cup \dots \cup X_D$ be the distance partition with respect to $x_0$ in $\Gamma$.
  Remark that $X_0 = \{x_0\}$ and $X_1 = \{x_1\}$.
  Then $X' = X_1 \cup \dots \cup X_D$ is a distance partition with respect to $x_1$ in $\Gamma'$.
  Thus $\Terw_2(\Gamma', x_1)$ is generated by $E_{X'} A E_{X'}$ and $E_{X_1}, \dots, E_{X_D}$.
  We have $E_{X'} \Terw_2(\Gamma, x_0)E_{X'} \supset \Terw_2(\Gamma', x_1)$.

  To show $E_{X'} \Terw_2(\Gamma, x_0) E_{X'} \subset \Terw_2(\Gamma', x_1)$,
  it is enough to show that
  \[
    E_{X'} A E_{X_{s_1}} A E_{X_{s_2}} \dots E_{X_{s_t}} A E_{X'} \in \Terw_2(\Gamma', x_1)
  \]
  for $0\leq s_u\leq D$.
  If $s_u\ne 0$ for some $u$, then we can divide the term into two parts.
  We may assume that $s_u=0$ for all $1\leq u\leq t$.
  However, $E_{X_0}AE_{X_0}=O$.
  Therefore, it is enough to consider $E_{X'}AE_{X'}$ and $E_{X'}AE_{x_0}AE_{X'}$.
  Now $E_{X'}AE_{X'}\in \Terw_2(\Gamma',x_1)$ and $E_{X'}AE_{x_0}AE_{X'}=E_{x_1}\in \Terw_2(\Gamma',x_1)$.
  We have $E_{X'}\Terw_2(\Gamma, x_0)E_{X'}\subset \Terw_2(\Gamma',x_1)$.

  We consider the case $\ell=3$.
  Set $G=\Aut(\Gamma)$.
  Since $G_{x_0}$ fixes $x_1$, we can see that the stabilizer of $x_1$ in $\Aut(\Gamma')$ coincides with $G_{x_0}$.
  Suppose that $X=Y_0\cup Y_1\cup \dots\cup Y_r$ is a partition of $X$
  into $G_{x_0}$-orbits where $Y_0=\{x_0\}$ and $Y_1=\{x_1\}$.
  Then $X'=Y_1\cup \dots\cup Y_r$ is a partition of $X'$ into $G_{x_0}$-orbits.
  Thus the same arguments as above can be applied and the statement for $\Terw_3$ holds.
\end{proof}

\begin{prop}\label{prop:valencyone}
  Suppose, for $\ell=2,3$, that $\Terw_\ell(\Gamma', x_1)\cong M_{n_0}(\C)\oplus\bigoplus_{i=1}^r M_{n_i}(\C)$,
  where the primitive idempotent $E_{x_1}$ belongs to $M_{n_0}(\C)$.
  Then $\Terw_\ell(\Gamma, x_0)\cong M_{n_0+1}(\C)\oplus\bigoplus_{i=1}^r M_{n_i}(\C)$
  and the primitive idempotent $E_{x_0}$ belongs to $M_{n_0+1}(\C)$.
\end{prop}

\begin{proof}
  Suppose $\ell$ is $2$ or $3$. 
  Let $I_{X'}=e_1+\dots+e_s$ be a primitive idempotent decomposition of $I_{X'}$ in $\Terw_\ell(\Gamma', x_1)$.
  By Lemma \ref{lem:equividemp} and Lemma \ref{lem:valencyone},
  $I_{X}=E_{x_0}+e_1+\dots+e_s$ is a primitive idempotent decomposition of $I_{X}$ in $\Terw_\ell(\Gamma, x_0)$.
  Since $E_{x_0}AE_{x_1}\ne 0$, we can say that $E_{x_0}$ and $E_{x_1}$ are equivalent.
  The assertion holds.
\end{proof}

\section{Some specific graphs}\label{sec:someGRPH}

In this section, we investigate the structures of $\Terw_1, \Terw_2, \Terw_3$, and $\Terw_4$ for path graphs, star graphs, and cycle graphs.

\subsection{Path graphs}

Let $P_n$ be the path graph with $n$ vertices for $n \geq 2$ \cite[Section 1.4.4]{MR2882891}.
We set the vertex set $\{1, 2, \dots, n\}$ and the edge set $E = \{\{i,j\}: |i-j| = 1\}$.
Let $m$ be the base vertex.
By symmetry, we may assume that $2m-1 \leq n$.

Let $\zeta$ be a primitive $(2n+2)$th root of unity in $\C$.
Then the eigenvalues of the adjacency matrix $A$ of the path graph $P_n$ are
\[
 \lambda_i := \zeta^{i} + \zeta^{-i} \quad (i = 1, \dots, n)
\]
and the corresponding eigenvectors are
\[
 {\bs v}_i := (\zeta^{i}-\zeta^{-i}, \zeta^{2i}-\zeta^{-2i}, \dots, \zeta^{ni}-\zeta^{-ni}).
\]
Moreover all eigenvalues are of multiplicity $1$.
Thus, we know $\Terw_0(P_n) \cong \bigoplus_{i=1}^{n} \C$ from Theorem \ref{basic0}.

First, we determine the structure of $\Terw_1$.

\begin{thm}
  $\Terw_1(P_n, m) \cong M_{n-t}(\C) \oplus \bigoplus_{i=1}^{t} \C$, where $t := \sharp \{s: 1 \leq s \leq n, (n+1)|sm\} = \lfloor \frac{n \gcd(m,n+1)}{n+1} \rfloor$.
Especially, $\dim \Terw_1(P_n, m) = (n-t)^2+t$.
\end{thm}

\begin{proof}
 Since $\C X = \bigoplus_{i=1}^{n} \C \bs{v}_i$, we can define projections $\pi_i : \C X \to \C X$ such that $\pi_i(\bs{v}_j) = \delta_{i,j} \bs{v}_i$.
Let $\varepsilon_i$ be the matrix corresponding to $\pi_i$, namely $\pi_i(\bs{w}) = \bs{w} \varepsilon_i$ for all $\bs{w} \in \C X$.
Then we know $\varepsilon_i \in \C \langle A \rangle = \Terw_0$ by theory of linear algebra.
Since $\lambda_i$ is a simple root, $\varepsilon_i$ has rank one and is a primitive idempotent in $\Terw_1$.

We will apply Proposition \ref{one1}.
The $m$-th entry of $\bs{v}_i$ is $\zeta^{mi}-\zeta^{-mi}$, and this is zero if and only if $2n+2\mid 2mi$.
Thus we have $t$ non-isomorphic $1$-dimensional simple $\Terw_1$-modules.

We consider the case of $\zeta^{mi}-\zeta^{-mi}\ne 0$.
We remark that $\C\bs{v}_i=\C X\varepsilon_i$.
Then $\bs{v}_i E_{x_0} \ne \bs{0}$ and hence $\varepsilon_i E_{x_0} \ne O$.
This means that $\varepsilon_i$ and  $E_{x_0}$ are equivalent primitive idempotents.
Thus the primitive idempotent decomposition $I_X=\sum_{i=1}^n \varepsilon_i$ has $(n-t)$ idempotents equivalent to $E_{x_0}$.
This means that the dimension of the simple $\Terw_1$-module $E_{x_0} \Terw_1$ is $n-t$.
The assertion holds.
\end{proof}

Next, we consider the structures of $\Terw_{\ell}(P_n, m)$ ($\ell = 2,3,4$) of the path graph $P_n$.
Let $X=X_0\cup\dots\cup X_D$ is the distance partition of $P_n$ with respect to the base vertex $m$.

\begin{lem}\label{lem_med2}
  The following statements hold.
  \begin{enumerate}[{\rm (1)}]
    \item $E_{m,m}\in \Terw_1(P_{n},m)$.
    \item $E_{m,m-k}+E_{m,m+k}\in \Terw_1(P_{n},m)$ for $k=1,\dots,m-1$.
    \item $E_{m-k,m}+E_{m+k,m}\in \Terw_1(P_{n},m)$ for $k=1,\dots,m-1$.
    \item $E_{m-k,m-k'}+E_{m-k,m+k'}+E_{m+k,m-k'}+E_{m+k,m+k'}\in \Terw_1(P_{n},m)$ for $k,k'=1,\dots,m-1$.
    \item $E_{m-k,m-k'}+E_{m+k,m+k'}\in \Terw_2(P_{n},m)$ for $k,k' =1,\dots,m-1$.
  \end{enumerate}
  Moreover, elements in the above statements are linearly independent.
\end{lem}

\begin{proof}
  The statement (1) is clear by definition.

  We prove (2) by the induction on $k$.
  First, we have $E_{m,m-1}+E_{m,m+1}=E_{X_0}A\in \Terw_1$.
  Suppose (2) holds for all $k'<k$.
  We have $\Terw_1\ni (E_{m,m-k+1}+E_{m,m+k-1})A=(E_{m,m-k+2}+E_{m,m+k-2})+(E_{m,m-k}+E_{m,m+k})$
  and $E_{m,m-k+2}+E_{m,m+k-2}\in \Terw_1$ by the inductive hypothesis.
  Thus $E_{m,m-k}+E_{m,m+k}\in \Terw_1$ holds.

  Similarly, (3) holds.

  By (2) and (3), we have
  $E_{m-k,m-k'}+E_{m-k,m+k'}+E_{m+k,m-k'}+E_{m+k,m+k'}=(E_{m+k,m}+E_{m-k,m})(E_{m,m-k'}+E_{m,m+k'})\in \Terw_1$
  and (4) holds.

  We have $\partial(m-k,m-k')= \partial(m+k,m+k')=|k-k'|$ and $\partial(m-k,m+k') = \partial(m+k,m-k')=|k+k'|>|k-k'|$.
  Thus $\Terw_2 \ni E_{X_k}A^{|k-k'|} E_{X_{k'}}=E_{m-k,m-k'}+E_{m+k,m+k'}$.
  (5) holds.

  It is easy to see these elements are linearly independent.
\end{proof}

\begin{thm}\label{Pn:neq}
  Suppose $n>2m-1$. Then $\Terw_2(P_n,m)=\Terw_3(P_n,m)=\Terw_4(P_n,m)=M_X(\C)$.
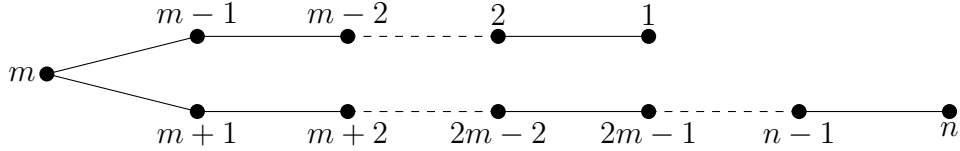
  \begin{center}
    \begin{tikzpicture}
      \coordinate[label=above:$1$]    (A) at (8,1);
      \coordinate[label=above:$2$]    (B) at (6,1);
      \coordinate[label=above:$m-2$]  (C) at (4,1);
      \coordinate[label=above:$m-1$]  (D) at (2,1);
      \coordinate[label=left:$m$]     (E) at (0,0.5);
      \coordinate[label=below:$m+1$]  (F) at (2,0);
      \coordinate[label=below:$m+2$]  (G) at (4,0);
      \coordinate[label=below:$2m-2$] (H) at (6,0);
      \coordinate[label=below:$2m-1$] (I) at (8,0);
      \coordinate[label=below:$n-1$]  (J) at (10,0);
      \coordinate[label=below:$n$]    (K) at (12,0);
      \draw (A)--(B);
      \draw[dashed] (B)--(C);
      \draw (C)--(D)--(E)--(F)--(G);
      \draw[dashed] (G)--(H);
      \draw (H)--(I);
      \draw[dashed] (I)--(J);
      \draw(J)--(K);
      \fill[black] (A) circle (0.1);
      \fill[black] (B) circle (0.1);
      \fill[black] (C) circle (0.1);
      \fill[black] (D) circle (0.1);
      \fill[black] (E) circle (0.1);
      \fill[black] (F) circle (0.1);
      \fill[black] (G) circle (0.1);
      \fill[black] (H) circle (0.1);
      \fill[black] (I) circle (0.1);
      \fill[black] (J) circle (0.1);
      \fill[black] (K) circle (0.1);
    \end{tikzpicture}
  \end{center}
\end{thm}

\begin{proof}
  It is enough to show that $\Terw_2=M_X(\C)$.
  Recall that $X = X_0 \cup \dots \cup X_{n-m}$ is the distance partition with respect to the vertex $m$.
  We have $\Terw_2\ni E_{X_0}=E_{m,m}$, $\Terw_2\ni E_{X_k} = E_{m-k,m-k} + E_{m+k,m+k}$ for $1 \leq k \leq m-1$, and $\Terw_2\ni E_{X_k} = E_{m+k,m+k}$ for $m \leq k \leq n-m$.
  For $1 \leq k \leq m-1$, we have 
\[
 E_{m+k,m+k} = E_{m+k,2m} E_{2m,m+k} = (E_{X_k}A^{m-k}E_{2m,2m})(E_{2m,2m}A^{m-k}E_{X_k}) \in \Terw_2.
\]
  Consequently, $E_{k,k}\in\Terw_2$ for all $1\leq k\leq n$.  

  Now, for every $1 \leq i, j \leq n$, we have $E_{i,i}A^{|i-j|}E_{j,j}\in \Terw_2$ is a nonzero multiple of $E_{i,j}$ and thus $E_{i,j}\in \Terw_2$.
  This leads to $\Terw_2=M_{X}(\C)$.
\end{proof}

Since the path graph $P_n$ is a tree and the stabilizer $G_{m}=1$ in the case of $n > 2m-1$, Theorem \ref{Pn:neq} holds also by \cite[Theorem 4]{MR4147627}.

\begin{thm}\label{prop_med}
  Suppose $n=2m-1$.
  Then $\Terw_2(P_{n},m)= \Terw_3(P_{n},m)= \Terw_4(P_{n},m)\cong M_{m}(\C)\oplus M_{m-1}(\C)$.
  Especially, $\dim \Terw_2(P_{n}, m)=\dim \Terw_3(P_{n}, m)=\dim \Terw_4(P_{n}, m) =2m^2-2m+1=(n^2+1)/4$.
  
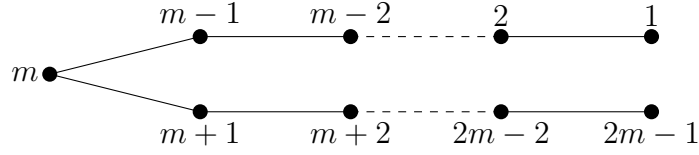
\begin{center}
    \begin{tikzpicture}
      \coordinate[label=above:$1$]    (A) at (8,1);
      \coordinate[label=above:$2$]    (B) at (6,1);
      \coordinate[label=above:$m-2$]  (C) at (4,1);
      \coordinate[label=above:$m-1$]  (D) at (2,1);
      \coordinate[label=left:$m$]     (E) at (0,0.5);
      \coordinate[label=below:$m+1$]  (F) at (2,0);
      \coordinate[label=below:$m+2$]  (G) at (4,0);
      \coordinate[label=below:$2m-2$] (H) at (6,0);
      \coordinate[label=below:$2m-1$] (I) at (8,0);
      \draw (A)--(B);
      \draw[dashed] (B)--(C);
      \draw (C)--(D)--(E)--(F)--(G);
      \draw[dashed] (G)--(H);
      \draw (H)--(I);
      \fill[black] (A) circle (0.1);
      \fill[black] (B) circle (0.1);
      \fill[black] (C) circle (0.1);
      \fill[black] (D) circle (0.1);
      \fill[black] (E) circle (0.1);
      \fill[black] (F) circle (0.1);
      \fill[black] (G) circle (0.1);
      \fill[black] (H) circle (0.1);
      \fill[black] (I) circle (0.1);
    \end{tikzpicture}
  \end{center}
\end{thm}

\begin{proof}
  The set of matrices in Lemma \ref{lem_med2} are in $\Terw_2$ and linearly independent.
  Thus $\dim\Terw_2 \geq 2m^2-2m+1$.

  Set $g=(1, n)(2, n-1)\dots(m-1, m+1)$.
  The stabilizer $G_m$ of $m$ in $G = \Aut(P_{n+1})$ is $G_m=G=\langle g \rangle$ of order $2$.
  For the permutation character $\rho$, we have $\rho(1)=2m-1$ and $\rho(g)=1$.
  Thus $\rho=m\ 1_{G_m}+(m-1)\chi_{G_m}$, where $1_{G_m}$ is the trivial character and $\chi_{G_m}$ is the unique non-trivial irreducible character of $G_m$.
  Thus $\Terw_4 \cong M_{m}(\C)\oplus M_{m-1}(\C)$.
  We have $\dim \Terw_2 \leq \dim \Terw_3 \leq \dim \Terw_4 = 2m^2-2m+1$.
\end{proof}

\subsection{Star graphs}

The \emph{star graph} with $n$ vertices is a complete bipartite graph $K_{1,n-1}$
\cite[p.~49]{MR1271140}, \cite[Section 1.4.2]{MR2882891}.
We set the vertex set $\{1, 2, \dots, n\}$ and the edge set $\{ \{1,i\}: i = 2, 3, \dots, n\}$.
By symmetry, it is enough to consider the case that the base vertex is $1$ or $2$.
Since the star graph is equal to the path graph $P_n$ if $n \leq 3$, we may assume that $n \geq 4$.
The eigenvalues of $K_{1,n-1}$ are $-\sqrt{n-1}, 0, \sqrt{n-1}$ with multiplicities $1, n-2, 1$, respectively.
Therefore we have $\Terw_0(K_{1,n-1}) \cong \C \oplus \C \oplus \C$ from Theorem \ref{basic0}.

\begin{thm}
 For $n \geq 4$, $\Terw_\ell(K_{1,n-1},1) \cong M_{2}(\C) \oplus \C$ for $\ell = 1,2,3,4$.
\end{thm}

\begin{proof}
  It is clear that $E_{1,1} \in \Terw_1$, $\sum_{i=2}^{n} E_{1,i} = E_{1,1} A \in \Terw_1$,
  $\sum_{i=2}^{n} E_{i,1} = A E_{1,1} \in \Terw_1$ and $\sum_{i=2}^{n} \sum_{j=2}^{n} E_{i,j} = (AE_{1,1})(E_{1,1}A) \in \Terw_1$.
Also we have $\sum_{i=2}^{n} E_{i,i} = I_{X} - E_{1,1} \in \Terw_1$. 
Thus, we know that $\dim \Terw_1 \geq 5$ since these matrices are linearly independent.

The stabilizer $G_{1}$ of $1$ in $G = \Aut(K_{1,n-1})$ is the symmetric group on $\{2, 3, \dots, n\}$.
Therefore the orbits of $G_1$ on $X \times X$ are $\{(1,1)\}, \{(1,i): i=2, \dots, n\}, \{(i,1): i=2,\dots,n\}, \{(i,i): i=2,\dots,n\}$, and $\{(i,j): i,j = 2,\dots,n, i \neq j\}$.
Thus $\dim \Terw_4 = 5$.
This means that $\Terw_1 = \Terw_2 = \Terw_3 = \Terw_4$.
Since $\Terw_\ell$ ($\ell = 1,2,3,4$) is a non-commutative $5$-dimensional semisimple $\C$-algebra, it follows that $\Terw_\ell \cong M_2(\C) \oplus \C$.
\end{proof}

\begin{thm}
  For $n \geq 4$,
  $\Terw_\ell(K_{1,n-1},2) \cong M_{3}(\C) \oplus \C$ for $\ell = 1,2,3,4$.
\end{thm}
\begin{proof}
 In this case, we can apply Theorem \ref{one4}. So we obtain that $\Terw_1 \cong M_{3}(\C) \oplus \C$ and $\dim \Terw_1 = 10$.

The stabilizer $G_2$ of $2$ in $G = \Aut(K_{1,n-1})$ is the symmetric group on $\{3, 4, \dots, n\}$,
 and the orbits of $G_2$ on $X \times X$ are $\{(1,1)\}$, $\{(1,2)\}$, $\{(2,1)\}$, $\{(2,2)\}$, $\{(1,i): i = 3, \dots, n\}$, $\{(i,1): i = 3, \dots, n\}$, $\{(2,i): i = 3, \dots, n\}$, $\{(i,2): i = 3, \dots, n\}$, 
$\{(i,i): i = 3, \dots, n\}$, and $\{(i,j): i, j = 3, \dots, n, i \neq j\}$.
Therefore we know that $\dim \Terw_4 = 10$ and $\Terw_1 = \Terw_2 = \Terw_3 = \Terw_4$.
\end{proof}

\subsection{Cycle graphs}

Let $C_n$ be the cycle graph with $n$ vertices ($n \geq 3$) \cite[Section 1.4.3]{MR2882891}.
The graph $C_n$ defines a $P$- and $Q$-polynomial association scheme
and the structure of the Terwilliger algebra $\Terw_2$ is determined in \cite{MR1229433}.
We set the vertex set $X = \{1, 2, \dots, n\}$ and the edge set $\{\{1,2\},\{2,3\}, \dots, \{n-1,n\}, \{n,1\}\}$.
By symmetry, it is enough to consider only the case that the base vertex is $1$.
Let $\xi$ be a primitive $n$-th root of unity in $\C$.
The diameter of $C_n$ is $ D= \lfloor \frac{n}{2} \rfloor$.
The eigenvalues of the adjacency matrix $A$ are $\lambda_i := \xi^i+\xi^{-i} \quad (i = 0, 1, \dots, D)$.
The multiplicity of $\lambda_0$ is $1$ and the multiplicity of $\lambda_i$ is $2$ for $i=1, \dots, D-1$.
The multiplicity of $\lambda_D$ is $1$ if $n$ is even and $2$ if $n$ is odd.

We remark that $C_n$ is a distance-regular graph and so can be applied Theorem \ref{one4}.

\begin{thm}
  $\Terw_0(\Gamma,x_0) \cong \bigoplus_{i=1}^{D+1} \C$ and 
\[
 \Terw_1(\Gamma,x_0) \cong \begin{cases}
		M_{D+1}(\C) \oplus \bigoplus_{i=1}^{D-1} \C & \text{if $n$ is even,} \\
		M_{D+1}(\C) \oplus \bigoplus_{i=1}^{D} \C & \text{if $n$ is odd,} \\
	       \end{cases}
\]
where $D = \lfloor \frac{n}{2} \rfloor$ is the diameter of $C_n$.
\end{thm}

\begin{proof}
It is clear from Theorem \ref{basic0} and Theorem \ref{one4}.
\end{proof}

Next, we consider the structures of $\Terw_2, \Terw_3$ and $\Terw_4$.

\begin{thm} \label{T_4}
  We have
\[
  \Terw_2(\Gamma,x_0) = \Terw_3(\Gamma,x_0) = \Terw_4(\Gamma,x_0) \cong \begin{cases}
    M_{D+1}(\C) \oplus M_{D-1}(\C) & \text{if $n$ is even,} \\
    M_{D+1}(\C) \oplus M_{D}(\C) & \text{if $n$ is odd,}
  \end{cases},
\]
where $D = \lfloor \frac{n}{2} \rfloor$ is the diameter of $C_n$.
\end{thm}

\begin{proof}
By \cite[Example 6.1 (23), (24)]{MR1229433},
we have that $\Terw_2 \cong M_{D+1}(\C) \oplus M_{D-1}(\C)$ if $n$ is even
and $\Terw_2 \cong M_{D+1}(\C) \oplus M_{D}(\C)$ if $n$ is odd.

We determine the structure of $\Terw_4$.
First, we consider the case of $n=2D$, namely $n$ is even.
Let $g = (2,2D)(3, 2D-1) \cdots (D,D+2)$.
Then the stabilizer $G_1$ of $1$ in the group $G = \Aut(C_n)$ is $\langle g \rangle$.
For the permutation character $\rho$, we know that $\rho(1) = 2D$ and $\rho(g) = 2$.
So $\rho = (D+1) 1_{G_1} + (D-1) \chi_{G_1}$,
where $1_{G_1}$ is the trivial character and $\chi_{G_1}$ is the non-trivial irreducible character of $G_1$.
Therefore, $\Terw_4 \cong M_{D+1}(\C) \oplus M_{D-1}(\C)$.

Next, we consider the case of $n=2D+1$.
Let $h = (2,2D+1)(3, 2D) \cdots (D+1,D+2)$.
Then the stabilizer $H_1$ of $1$ in $H = \Aut(C_n)$ is $\langle h \rangle$.
Let $\rho$ be the permutation character.
Since it follows that $\rho(1) = 2D+1$ and $\rho(h) = 1$,
we have $\rho = (D+1) 1_{H_1} + D \varphi_{H_1}$,
where $1_{H_1}$ is the trivial character and $\varphi_{H_1}$ is the non-trivial irreducible character of $H_1$.
Thus $\Terw_4 \cong M_{D+1}(\C) \oplus M_{D}(\C)$.
\end{proof}

\section{Paley graphs: Examples for $\Terw_3 \neq \Terw_4$}\label{sec:Paley}

In this section, we consider Paley graphs \cite[page 129]{MR1271140}, \cite{Jones-Paley}.
We show in Corollary \ref{ex:paley} that Paley graphs with $p$ vertices ($p$ is a prime such that $p \geq 7$) give examples for $\Terw_3 \neq \Terw_4$.

Let $\Paley(p^a)=(X,E)$ be the Paley graph with $|X|=p^a\equiv 1\pmod{4}$,
where $p$ is a prime number.
It is a strongly regular graph, and thus we can apply Theorem \ref{one4} for $\Terw_1$.

We describe the automorphism group $G=\Aut(\Paley(p^a))$.
The vertex set $X$ is identified with the finite field $\GF(p^a)$.
We fix a primitive element $\xi$ of $\GF(p^a)$.
The group $G$ is generated by $\mu_\alpha:x\mapsto x+\alpha$ ($\alpha\in \GF(p^a)$),
$\sigma:x\mapsto x\xi^2$, and $\tau:x\mapsto x^p$.
The orders of $\sigma$ and $\tau$ are $(p^a-1)/2$ and $a$, respectively.
The stabilizer of $0$ in $G$ is $G_0=\langle \sigma,\tau\rangle$.
The Paley graph is distance-transitive and so the structure of $\Terw_\ell(\Paley(p^a), x_0)$ ($\ell=0,1,2,3,4$)
does not depend on a choice of the base vertex $x_0$.
Thus we fix the base vertex $0$ and
write $\Terw_\ell(\Paley(p^a))$ for $\Terw_\ell(\Paley(p^a), 0)$.
We keep these notations in this section.

\begin{thm}\label{paley1}
  $\Terw_1(\Paley(p^a))\cong M_3(\C) \oplus \C \oplus \C$ and $\dim \Terw_1(\Paley(p^a))=11$.
\end{thm}

\begin{proof}
  The graph $\Paley(p^a)$ has eigenvalues $(p^a-1)/2$, $(-1+\sqrt{p^a})/2$, $(-1-\sqrt{p^a})/2$
  with multiplicities $1$, $(p^a-1)/2$, $(p^a-1)/2$, respectively.
  By Theorem \ref{one4}, we have the result.
\end{proof}

\begin{thm}\label{paley2}
  $\dim \Terw_4(\Paley(p^a))\leq 2p^a+3$.
  For the case $|X|=p$, we have the equality
  $\dim \Terw_4(\Paley(p))= 2p+3$ and 
  $\Terw_4(\Paley(p))\cong M_3(\C)\oplus\bigoplus_{i=1}^{(p-3)/2} M_{2}(\C)$.
\end{thm}

\begin{proof}
  Set $k=(p^a-1)/2$.
  We consider the centralizer algebra $\mathcal{U}$ of $H=\langle\sigma\rangle\subset G_0$.
  Then $\Terw_4$ is a subalgebra of $\mathcal{U}$.
  We fix an ordering of $X=\GF(p^a)$ by
  $$\{0,\quad 1,\xi^2,\xi^4,\dots,\xi^{p^a-3},\quad \xi,\xi^3,\dots,\xi^{p^a-2}\}.$$
  Then the permutation matrix given by $\sigma$ with respect to this ordering is
  $$\left(
    \begin{array}{c|cccc|cccc}
      1 &  & & & & & & & \\
      \hline
        & 0 &1 & & & & & &\\
        &&\ddots&\ddots&&&&&\\
        &&&\ddots&1&&&&\\
        &1&&&0&&&&\\
      \hline
       &&&&&0&1&&\\
       &&&&&&\ddots&\ddots&\\
       &&&&&&&\ddots&1\\
       &&&&&1&&&0
    \end{array}
  \right).$$
  Thus the elements in $\mathcal{U}$ are of the form
  \begin{equation}\label{eq:circulant}
    \left(
    \begin{array}{c|cccc|cccc}
      a & b & \dots & \dots & b & c & \dots  & \dots & c\\
      \hline
      b'  & d_0 &d_1 & \dots & d_{k-1}& e_0& e_1& \dots&e_{k-1}\\
      \vdots  &d_{k-1}&\ddots&\ddots&\vdots&e_{k-1}&\ddots&\ddots&\vdots\\
      \vdots  &\vdots&\ddots&\ddots&d_1&\vdots&\ddots&\ddots&e_1\\
      b'  &d_1&\dots&d_{k-1}&d_0&e_1&\dots&e_{k-1}&e_0\\
      \hline
      c' &e_0'&e_1'&\dots&e_{k-1}'&d_0'&d_1'&\dots&d_{k-1}'\\
      \vdots  &e_{k-1}'&\ddots&\ddots&\vdots&d_{k-1}'&\ddots&\ddots&\vdots\\
      \vdots &\vdots&\ddots&\ddots&e_1'&\vdots&\ddots&\ddots&d_1'\\
      c' &e_1'&\dots&e_{k-1}'&e_0'&d_1'&\dots&d_{k-1}'&d_0'
    \end{array}
  \right).
  \end{equation}
  The number of parameters is $5+4k=2p^a+3$ and this is the dimension of $\mathcal{U}$.
  Thus $\dim \Terw_4\leq 2p^a+3$.
  
  Suppose $|X|=p$. In this case, $G_0=H$ and $\mathcal{U}=\Terw_4$, and thus the equality holds.
  Remark that all $E_{X_i}\Terw_4E_{X_j}$ ($i,j=1,2$) are isomorphic to the group algebra of the cyclic group of order $k$,
  though $E_{X_i}\Terw_4E_{X_j}$ are not algebras if $i\ne j$ but we identify them.
  Let $\varepsilon_0,\dots,\varepsilon_{k-1}$ be the primitive idempotent of the group algebra,
  where $\varepsilon_0$ corresponds to the trivial representation, and set
  $$\mathcal{A}_0=
  \left\{\left(\begin{array}{c|ccc|ccc}
          a&b&\dots&b&c&\dots&c\\
          \hline
          b'&&&&&&\\
          \vdots&&d \varepsilon_0&&& e  \varepsilon_0&\\
          b'&&&&&&\\
          \hline
          c'&&&&&&\\
          \vdots&&e' \varepsilon_0&&& d'  \varepsilon_0&\\
          c'&&&&&&\\
        \end{array}\right)\right\},\ 
      \mathcal{A}_s=
      \left\{\left(\begin{array}{c|ccc|ccc}
              0&0&\dots&0&0&\dots&0\\
              \hline
              0&&&&&&\\
              \vdots&&d \varepsilon_s&&& e  \varepsilon_s&\\
              0&&&&&&\\
              \hline
              0&&&&&&\\
              \vdots&&e' \varepsilon_s&&& d'  \varepsilon_s&\\
              0&&&&&&\\
            \end{array}\right)\right\},$$
  for $1\leq s\leq k-1$.
  Then $\mathcal{U}=\bigoplus_{s=0}^{k-1}\mathcal{A}_s$ as algebras, and
  $\mathcal{A}_0\cong M_3(\C)$, $\mathcal{A}_s\cong M_2(\C)$ for $1\leq s\leq k-1$.
  We have $\Terw_4\cong M_3(\C)\oplus\bigoplus_{i=1}^{(p-3)/2} M_{2}(\C)$.
\end{proof}

\begin{remark}
  The equality in Theorem \ref{paley2} does not hold for $p^a$, $a>1$. For example, 
  $$\begin{array}{c||c|c|c|c|c|c}
      p^a & 9 & 25 & 49 & 81 & 121 & 125 \\
      \hline
      \dim\Terw_4 & 15 & 33 & 59 & 51 & 135 & 93\\
    \end{array}$$
\end{remark}

\begin{thm}\label{paley3}
  $\Terw_2(\Paley(p^a))=\Terw_3(\Paley(p^a))$.
\end{thm}

\begin{proof}
  Since $\Paley(p^a)$ is distance-transitive, the statement holds by Lemma \ref{basic3}.
\end{proof}

\begin{thm}\label{paley4}
  $\dim \Terw_2(\Paley(p^a))\leq p^a+8$.
\end{thm}

\begin{proof}
  By Proposition \ref{SRG}, $E_{X_1}\Terw_2E_{X_1}$ is commutative and
  generated by symmetric matrices.
  Thus all matrices in $E_{X_1}\Terw_2E_{X_1}$ are symmetric.
  By the form (\ref{eq:circulant}) of matrices, $E_{X_1}\Terw_2E_{X_1}$ is
  contained in the adjacency algebra of the cycle $C_{(p^a-1)/2}$.
  Therefore $d :=\dim E_{X_1}\Terw_2E_{X_1}\leq (p^a+3)/4$.
  The algebra $E_{X_1}\Terw_2E_{X_1}$ has $d$ non-isomorphic $1$-dimensional simple modules.
  The idempotent $E_{X_1}$ is a sum of $d$ non-equivalent primitive idempotents in $E_{X_1}\Terw_2E_{X_1}$.
  By Lemma \ref{lem:equividemp},
  $E_{X_1}$ is a sum of $d$ non-equivalent primitive idempotents in $\Terw_2$,
  and the same thing holds for $E_{X_2}$.
  Now $E_{X_1}\Terw_2 E_{X_2}\cong \Hom_{\Terw_2}(E_{X_1}\Terw_2, E_{X_2}\Terw_2)$
  and $E_{X_1}\Terw_2$ and $E_{X_2}\Terw_2$ are sums of $d$ non-isomorphic simple $\Terw_2$-modules.
  We have $\dim E_{X_1}\Terw_2 E_{X_2}\leq d$ and similarly $\dim E_{X_2}\Terw_2 E_{X_1}\leq d$.
  Now we can conclude that
  $$\dim \Terw_2=\sum_{i=0}^2 \sum_{j=0}^2 \dim E_{X_i}\Terw_2 E_{X_j}
  \leq 5+4d \leq p^a+8$$
  and the result holds.
\end{proof}

\begin{example}
  \begin{enumerate}[(1)]
    \item $\dim \Terw_2=p+8$ holds for $|X|=p=5$, $13$, $17$, $29$, $41$, $53$, $89$, $109$, $113$, $137$, $149$.
    \item $\dim \Terw_2=p+4$ holds for $|X|=p=37$, $61$, $73$, $97$, $101$.
    \item
    $\dim \Terw_2=p^a+6$ for $p^a=9$, 
    $\dim \Terw_2=p^a$ for $p^a=25$, 
    $\dim \Terw_2=p^a-14$ for $p^a=49$, 
    $\dim \Terw_2=p^a-48$ for $p^a=81$, 
    $\dim \Terw_2=p^a-54$ for $p^a=121$, 
    $\dim \Terw_2=p^a-72$ for $p^a=125$, 
  \end{enumerate}
\end{example}

\begin{cor} \label{ex:paley}
  For $\Paley(p)$ with a prime number $p\geq 7$,
  $\Terw_1(\Paley(p)) \subsetneq \Terw_2(\Paley(p)) = \Terw_3(\Paley(p)) \subsetneq \Terw_4(\Paley(p))$.
\end{cor}

\begin{proof}
  This is clear by Theorems \ref{paley1}, \ref{paley2}, \ref{paley3}, \ref{paley4}. 
\end{proof}

\section{Examples for $\Terw_2 \neq \Terw_3$}\label{sec:Gamma}
We consider the following graph $\Delta_n$ with $n$ vertices for $n\geq 5$.
The graph $\Delta_5$ with the base vertex $5$ is the example of the minimum vertices for $\Terw_2 \neq \Terw_3$.
(We calculated by using McKay's database \cite{McKay} for connected simple graphs and GAP4 \cite{GAP4}).

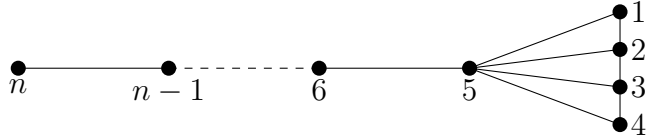
\begin{center}
  \begin{tikzpicture}
    \coordinate[label=below:$n$]   (A) at (0,0.75);
    \coordinate[label=below:$n-1$] (B) at (2,0.75);
    \coordinate[label=below:$6$]   (C) at (4,0.75);
    \coordinate[label=below:$5$]   (D) at (6,0.75);
    \coordinate[label=right:$4$]   (E) at (8,0);
    \coordinate[label=right:$3$]   (F) at (8,0.5);
    \coordinate[label=right:$2$]   (G) at (8,1);
    \coordinate[label=right:$1$]   (H) at (8,1.5);
    \draw (E)--(F)--(G)--(H);
    \draw (H)--(D)--(E);
    \draw (G)--(D)--(F);
    \draw (C)--(D);
    \draw[dashed] (B)--(C);
    \draw (A)--(B);
    \fill[black] (A) circle (0.1);
    \fill[black] (B) circle (0.1);
    \fill[black] (C) circle (0.1);
    \fill[black] (D) circle (0.1);
    \fill[black] (E) circle (0.1);
    \fill[black] (F) circle (0.1);
    \fill[black] (G) circle (0.1);
    \fill[black] (H) circle (0.1);
  \end{tikzpicture}
\end{center}
We would like to show that $\Terw_2(\Delta_n, n)\subsetneq \Terw_3(\Delta_n, n)$.
By Proposition \ref{prop:valencyone}, it is enough to show it for the case $n=5$.

We set $\Terw_\ell:=\Terw_\ell(\Delta_5, 5)$ for $\ell=0,1,2,3,4$,
and $X_0:=\{5\}$, $X_1:=\{1,2,3,4\}$.
Obviously $\Terw_1=\Terw_2$ holds.
The adjacency matrix of $\Delta_5$ is
$$\left(
  \begin{array}{cccc|c}
    0&1&0&0&1\\
    1&0&1&0&1\\
    0&1&0&1&1\\
    0&0&1&0&1\\
    \hline
    1&1&1&1&0
  \end{array}
\right).$$
We set
$$
B_1=\left(
  \begin{array}{c|c}
    O_4 & {}^t \bs{0}_4\\
    \hline
    \bs{0}_4 & 1
  \end{array}
\right),\
B_i=\left(
  \begin{array}{c|c}
    O_4 & {}^t \bs{v}_i\\
    \hline
    \bs{0}_4 & 0
  \end{array}
\right)\ (i=2,3),\ 
B_i=B_{i-2}^T\ (i=4,5),$$
and
$$
B_i=\left(
  \begin{array}{c|c}
    B_{i}' & {}^t \bs{0}_4\\
    \hline
    \bs{0}_4 & 0
  \end{array}
\right)\ (i=6,\dots,13),
$$
where
$\bs{0}_4=(0,0,0,0)$, $O_4$ is the zero matrix of degree $4$,
$\bs{v}_2=(1,1,1,1)$, $\bs{v}_3=(0,1,1,0)$, and
$$
B_6'=\left(\begin{array}{cccc}
1&0&0&0\\
0&1&0&0\\
0&0&1&0\\
0&0&0&1
\end{array}\right),\ 
B_7'=\left(\begin{array}{cccc}
0&1&0&0\\
1&0&1&0\\
0&1&0&1\\
0&0&1&0
\end{array}\right),\ 
B_8'=\left(\begin{array}{cccc}
0&0&1&0\\
0&1&0&1\\
1&0&1&0\\
0&1&0&0
\end{array}\right),$$
$$
B_9'=\left(\begin{array}{cccc}
0&0&0&1\\
0&0&1&0\\
0&1&0&0\\
1&0&0&0
\end{array}\right),\ 
B_{10}'=\left(\begin{array}{cccc}
0&0&0&0\\
0&1&1&0\\
0&1&1&0\\
0&0&0&0
\end{array}\right),\ 
B_{11}'=\left(\begin{array}{cccc}
0&1&1&0\\
0&0&0&0\\
0&0&0&0\\
0&1&1&0
\end{array}\right),$$
$$
B_{12}'=\left(\begin{array}{cccc}
0&0&0&0\\
0&1&0&0\\
0&0&1&0\\
0&0&0&0
\end{array}\right),\ 
B_{13}'=\left(\begin{array}{cccc}
0&1&0&0\\
0&0&0&0\\
0&0&0&0\\
0&0&1&0
\end{array}\right).$$
Direct calculation shows the following.

\begin{lem}\label{lem:delta1}
  The set $\{B_1,\dots,B_{11}\}$ is a basis of $\Terw_2(\Delta_5, 5)$, and
  $\{B_1,\dots,B_{13}\}$ is a basis of $\Terw_3(\Delta_5, 5)$.
  Especially, $\dim \Terw_2(\Delta_5, 5)=11<13=\dim \Terw_3(\Delta_5, 5)$.
\end{lem}

\begin{lem}\label{lem:delta2}
  $\Terw_1(\Delta_5, 5)=\Terw_2(\Delta_5, 5)\cong M_3(\C)\oplus \C \oplus \C$
  and $\Terw_3(\Delta_5, 5)\cong M_3(\C)\oplus M_2(\C)$.
\end{lem}

\begin{proof}
  The algebra $E_{X_1}\Terw_2 E_{X_1}$ contains the adjacency algebra of the path graph of length $4$.
  Thus $E_{X_1}$ is decomposed into four primitive idempotents
  $E_{X_1}=e_1+e_2+e_3+e_4$
  in $E_{X_1}\Terw_2 E_{X_1}$.
  and thus the identity matrix $I$ is decomposed into five primitive idempotents
  $I=E_{X_0}+e_1+e_2+e_3+e_4$ in $\Terw_2$.
  Since $\Hom_{\Terw_2}(E_{X_0}\Terw_2, E_{X_1}\Terw_2)\cong E_{X_0}\Terw_2 E_{X_1}=\C B_2\oplus \C B_3$
  has dimension $2$, exactly two of $\{e_1,e_2,e_3,e_4\}$, say $e_1$, $e_2$,
  are equivalent to $E_{X_0}$.
  Thus $\Terw_2\cong M_3(\C)\oplus \C \oplus \C$ if $e_3$ and $e_4$ are non-equivalent,
  or $\Terw_2\cong M_3(\C)\oplus M_2(\C)$ if $e_3$ and $e_4$ are equivalent.
  However, we know that $\dim \Terw_2=11$.
  We have the result for $\Terw_2$.

  By the same argument shows $\Terw_3\cong M_3(\C)\oplus M_2(\C)$.
\end{proof}

Now we can determine $\Terw_\ell(\Delta_n, n)$, $\ell=1,2,3,4$.

\begin{thm}\label{thm:delta3}
  For 
  $n\geq 5$,
  $\Terw_1(\Delta_n, n) = \Terw_2(\Delta_n, n)\cong M_{n-2}(\C)\oplus\C \oplus \C$ and
  $\Terw_3(\Delta_n, n)=\Terw_4(\Delta_n, n)\cong M_{n-2}(\C)\oplus M_2(\C)$.
  Especially, $\dim \Terw_1(\Delta_n, n) = \dim \Terw_2(\Delta_n, n)=n^2-4n+6<n^2-4n+8
  =\dim \Terw_3(\Delta_n, n)=\dim \Terw_4(\Delta_n, n)$ holds.
\end{thm}

\begin{proof}
  By Proposition \ref{prop:valencyone} and Lemma \ref{lem:delta2}, we have the results for
  $\Terw_2(\Delta_n, n)$ and $\Terw_3(\Delta_n, n)$.
  By induction, we can prove that $E_k \in \Terw_1(\Delta_n, n)$ for $5 \leq k \leq n$.
Thus $\Terw_1(\Delta_n,n) = \Terw_2(\Delta_n, n)$.
  The automorphism group of $\Delta_n$ is a cyclic group of order $2$.
  By the similar argument to the proof of Theorem \ref{prop_med} for the permutation character,
  we have the result for $\Terw_4(\Delta_n, n)$.
\end{proof}

\begin{rem}
 We can find no examples for $\Terw_3 \neq \Terw_4$ among all connected simple graphs with vertices less than or equal to $9$ (by using the McKay's database \cite{McKay} and GAP4 \cite{GAP4}).
Thus we find no examples for $\Terw_2 \neq \Terw_3 \neq \Terw_4$.
\end{rem}


\bibliographystyle{amsplain}
\bibliography{genterw}
\end{document}